\documentclass[3p]{elsarticle}
\usepackage{amsmath}
\usepackage{amssymb}
\usepackage{hyperref}
\usepackage{color}

\newtheorem{theorem}{Theorem}
\newtheorem{claim}{Claim}
\newtheorem{corollary}{Corollary}
\newdefinition{definition}{Definition}
\newproof{example}{\textbf{Example}}
\newtheorem{lemma}{Lemma}
\newproof{notation}{{\bf Notation}}
\newtheorem{proposition}{Proposition}
\newdefinition{remark}{Remark}
\newproof{proof}{{\bf Proof}}
\newproof{assumption}{{\bf Assumptions}}
\newproof{Space}{{\bf Function Spaces}}

\numberwithin{equation}{section}

\newcommand{\I}{\infty}
\renewcommand{\S}{\sigma}

\newcommand{\N}{\nabla}

\renewcommand{\L}{\lambda}

\renewcommand{\r}{\varrho}

\renewcommand{\P}{\partial}
\newcommand{\B}{\beta}
\newcommand{\D}{\delta}
\newcommand{\Lap}{\triangle}

\newcommand{\A}{\alpha}
\newcommand{\lng}{\langle}
\newcommand{\rng}{\rangle}
\newcommand{\omg}{\omega}
\newcommand{\V}{\varepsilon}

\newcommand{\QED}{\qquad\mbox{\qed}}

\newcommand{\R}{\mathbb{R}}
\newcommand{\BN}{\mathbb{N}}

\newcommand{\CB}{\mathcal{B}}
\newcommand{\CG}{\mathcal{G}}

\newcommand{\CF}{\mathcal{F}}
\newcommand{\CL}{\mathcal{L}}

\newcommand{\CJ}{\mathcal{J}}
\newcommand{\CK}{\mathcal{K}}

\newcommand{\CS}{\mathcal{S}}

\renewcommand{\CG}{\mathcal{G}}

\journal{}

\begin{document}

\begin{frontmatter}
\title{Nonlocal equations with regular varying decay solutions}

\author[chu]{Sujin Khomrutai\corref{cor1}}
\ead{sujin.k@chula.ac.th}

\address[chu]{Department of Mathematics and Computer Science, Faculty of Science, Chulalongkorn University, Bangkok 10330, Thailand}

\cortext[cor1]{Corresponding author.}


\begin{abstract}
We study the asymptotic behavior for nonlocal diffusion equations $\partial_tu=\CJ u-\chi_0u$ in $\R^n\times(0,\infty)$ and obtain a sufficient condition so that solutions of the Cauchy problem decay in time at the rate of a regular varying function. 
In the sufficient condition, a sharp bound of certain forms is required for the $k$-fold iterations $\CJ^ku_0$ or the kernels $J_k$.
We prove the desired decay rate by analyzing the asymptotic behavior of a regular varying modified exponential series. Then we verify that the sufficient condition is true for most of the known radially symmetric kernels, and for some more general kernels, using the sharp Young's convolution inequality and a Fourier splitting argument. Classical results on the decay of solutions for these nonlocal diffusion equations are re-established and generalized. Finally, using our framework, we can exhibit a kernel having a prescribed regular varying decay solutions for a wide class of regular varying functions.

\end{abstract}

\end{frontmatter}

\section{Introduction}

In this work, we give a sufficient condition for solutions of the nonlocal equation
\begin{align}\label{Eqn:main}
\P_tu=\int_{\R^n}J(x,y)u(y,t)dy-\chi_0u(x,t)\quad(x,t)\in\R^n\times(0,\I),
\end{align}
to decay in time at the rate of a regular varying function. Here $J=J(x,y)$ is a given function, not necessarily radially symmetric, and $\chi_0>0$ is a constant. The nonlocal equations of this form have been used to model and study many phenomena such as diffusion, image enhancement \cite{GilboaEtal08}, phase transition \cite{BatesEtal99}, dispersal of a species by a long-range effects \cite{Fife03}, etc. See also \cite{AndreuEtal10} and the reference therein.\medskip

In the first step of our investigation, we express the solution of (\ref{Eqn:main}) as a power series in time involving the $k$-folds iterations
\[
\CJ^k=\CJ\circ\cdots\circ\CJ\quad\mbox{($k$ terms of $\CJ$), acting on the initial condition $u_0=u|_{t=0}$},
\]
where $\CJ$ is the integral operator
\[
\CJ u_0(x)=\int_{\R^n}J(x,y)u_0(y)dy.
\]
Indeed, we have the representation formula for solution of (\ref{Eqn:main}) as
\[
u(t)=e^{-\chi_0t}u_0+e^{-\chi_0t}\sum_{k=1}^\infty\frac{t^k}{k!}\CJ^ku_0.
\]

Then we turn the investigation into bounding norms of $J_k$, the kernels of the operators $\CJ^k$, or norms of the functions $\CJ^ku_0$. To the author knowledge, there have been no studies of nonlocal equations 
in this direction, where the asymptotic behavior of solutions 
is derived directly from the asymptotic behavior
 of $J_k$ or of $\CJ^ku_0$ as $k\to\infty$. (Although, the closest one is \cite{BrandleEtal11}, where compact support and Gaussian $J$ are considered.) The benefit of taking this approach is that the results can be applied to nonlocal 
equations with real- or complex-valued kernels. This is in contrast with many works on asymptotic behavior of nonlocal equations that rely heavily on the positivity of the kernel. In fact, the positivity enables the application of comparison and barrier arguments. Another possible benefit of this approach is that it could lead to the study of non-symmetric nonlocal equations.\medskip

At an abstract level, we can prove the following general result. Assume that all $J_k$ or all $\CJ^ku_0$ lie in a Banach space $X$ with norm $\|\cdot\|$, and an estimate of the form
\begin{align}\label{Tmp:JkJku0}
\|J_k\|\leq R_k\quad\mbox{or}\quad\|\CJ^ku_0\|\leq R_k
\end{align}
respectively, holds for all $k$ sufficiently large, where 
\[
R_k=R(k)=k^{\beta}L(k)\quad(\beta\in\R)
\]
is a regular varying function. 
Then we are able to prove that the solution of (\ref{Eqn:main}) satisfies the asymptotic behavior
\[
\|u(t)\|\lesssim t^{\beta}L(t)\quad\mbox{as $t\to\infty$}
\]
in some suitable Banach space $Y$. This abstract result is valid for integral operators which can be either symmetric or non-symmetric, real-valued, or complex-valued. We obtained the preceding result by establishing the asymptotic behavior of the exponential type power series
\[
\sum_{k=N}^\infty\frac{(\alpha t)^k}{k!}R_k\asymp R(\alpha t)e^{\alpha t}\quad\mbox{as $t\to\infty$}.
\]

Having the above abstract result, we are now facing a new challenging question. For a given kernel $J$, how do we get an inequality of the form (\ref{Tmp:JkJku0})? In this work we pursue this question for radially symmetric kernels, i.e.\ $J=J(x-y)$. Note that in this case the $k$-fold product kernel function is
\[
J_k=J\ast\cdots\ast J,\quad\mbox{the $(k-1)$-times convolution}.
\]
The Banach spaces $X,Y$ are $L^p(\R^n)$, where $1\leq p\leq\infty$. Note that the usual Young's convolution inequality is not enough to get ``good bounds $R_k$" for $\|J_k\|_{L^p}$ or $\|\CJ^ku_0\|_{L^p}$, in the sense that the resulting power series for the solution does not exhibit a power decay in time, especially when $\chi_0=\|J\|_{L^1}$. Therefore some more sophisticated tools have to be employed.\medskip

The simplest convolution integral operators considered in this work are those having kernels possessing a higher integrability: $J\in L^1(\R^n)\cap L^r(\R^n)$ for some $r>1$. Such operators or kernels include
\begin{itemize}
\item[(1)] Continuous function with compact support, 
\item[(2)] $(1-\Lap)^{-1}$ (the Bessel potential operator), 
\item[(3)] $(\lambda-\CL)^{-1}$, where $\lambda>0$ and $\CL$ is an elliptic operator, 
\item[(4)] weakly singular operator, etc.
\end{itemize}
For these kernels, we employ the sharp Young's convolution inequality to show  that
\[
\|J_k\|_{L^\infty}\lesssim k^{-n/2}\quad\forall\, k\,\,\mbox{large}
\]
The sharp constant in the sharp Young's (or Brascamp-Lieb) inequality played a crucial role in getting this asymptotic bound. After establishing this fundamental result, we can apply the abstract result from the previous paragraph to get the asymptotic behavior of solutions of (\ref{Eqn:main}) when $u_0\in L^1(\R^n)\cap L^\infty(\R^n)$. For the initial condition $u_0\in L^1(\R^n)$, the proof of our result directly give a refined asymptotic behavior generalizing partly the corresponding result 
in \cite{IgnatEtal08}.\medskip

Next, we turn our study to the stable laws. For simplicity, we put $\chi_0=\|J\|_{L^1}=1$ in (\ref{Eqn:main}) and $J\geq0$. We assume in this case that the kernel has the expansion in the Fourier variables as
\[
\widehat{J}(\xi)=1-A|\xi|^\sigma(\ln(1/|\xi|)^\mu+o(|\xi|^\sigma(\ln(1/|\xi|)^\mu)\quad\mbox{as $|\xi|\to0$}.
\]
In the special case that
\[
\mu=0,\quad\mbox{or}\quad\mu=1,
\]
the result we obtained are classical results in \cite{ChasseigneEtal06}. So we have generalized the results to all real number $\mu$. For stable laws with $0<\sigma<2$, the kernels possess no higher
integrability: $\|J\|_{L^r}=\infty$ for all $r>1$ (see \cite{Feller71}). This means we cannot apply the results from the previous case. To compensate this difficulty, we analyze the functions $\CJ^ku_0$ instead of the kernels $J_k$. As in \cite{ChasseigneEtal06}, an integrability assumption on $u_0$ and its Fourier transform $\widehat{u}_0$ have to be made. Now thanks to the radial symmetry of $J$, we get that
\[
\widehat{J}_k(\xi)=\widehat{J}(\xi)^k\quad\forall\,k\in\BN.
\]
Then employing a Fourier splitting of argument on the frequency domain $\R^n$, we can prove a bound for 
\[
\|\CJ_ku_0\|_{L^\infty}\lesssim(k\ln k)^{-n/\sigma}\quad\mbox{as $t\to\infty$},
\]
and then the asymptotic behavior of solutions of (\ref{Eqn:main}) follows directly from the abstract result.\medskip

Finally, we extend the work to arbitrary slowly varying function $L:(0,\infty)\to(0,\infty)$ and arbitrary $\beta>0$. Under an assumption on $L$, we exhibit a kernel $J$ such that the solutions to (\ref{Eqn:main}) satisfy
\[
\|u(t)\|_{L^p}\lesssim(tL(t))^{-\beta}\quad\mbox{as $t\to\infty$}.
\]

\section{Preliminaries}\label{Sec:prelim}

\subsection*{a. Notation, basic fact, and convention}

Let $\varGamma(s)=\int_0^\infty e^{-\tau}\tau^{s-1}d\tau$ be the Gamma function and 
\[
\CF\{f\}(\xi)=\widehat{f}(\xi)=\int_{\R^n}f(x)e^{-ix\cdot\xi}dx
\]
the Fourier transform of $f$. We denote
\begin{align*}
&a_k\sim b_k\,\,\,\,\mbox{as $k\to\infty$}\quad\Leftrightarrow\quad\lim_{k\to\infty}\frac{a_k}{b_k}=1\\
&f(t)\sim g(t)\,\,\,\,\mbox{as $t\to\infty$}\quad\Leftrightarrow\quad\lim_{t\to\infty}\frac{f(t)}{g(t)}=1.
\end{align*}
We shall often use the fact that if two sequences $\{a_k\},\{b_k\}$ satisfy $a_k\sim b_k$ as $k\to\infty$ and $b_k\neq0$ for all $k$ large, then there is a constant $C>0$ such that
\[
\frac{1}{C}b_k\leq a_k\leq Cb_k\quad\forall\,k\geq k_0.
\]
If two functions $f(t)\sim g(t)$ as $t\to\infty$ and $g(t)\neq0$ for all $t$ large, then there is a constant $C>0$ such that
\[
\frac{1}{C}g(t)\leq f(t)\leq Cg(t)\quad\forall\,t\geq t_0,
\]
i.e.\ $f(t)\asymp g(t)$ as $t\to\infty$.\medskip

%\medskip

Throughout this work $J=J(x,y)$ is a function defined for $(x,y)\in\R^n\times\R^n$, $J$ may be complex-valued, and let $\CJ$ be the integral operator with kernel $J$, that is
\[
\CJ u:=\int_{\R^n}J(x,y)u(y)dy.
\]
For each positive integer $k$, the $k$-fold product $\CJ^k=\CJ\circ\cdots\circ\CJ$ ($k$ terms of $\CJ$) is the integral operator whose kernel $J_k=J_k(x,y)$ is given by
\begin{align*}
&J_k(x,y)=\int_{(\R^n)^{k-1}}J(x,y_1)J(y_1,y_2)\cdots J(y_{k-1},y)dy_{k-1}\cdots dy_1.
\end{align*}
Thus
\begin{align*}
\CJ^k u(x)&=\int_{\R^n}J_k(x,y)u(y)dy\\
&=\int_{(\R^n)^k}J(x,y_1)J(y_1,y_2)\cdots J(y_{k-1},y)u(y)dydy_{k-1}\cdots dy_1.
\end{align*}
Note that, if the kernel is radially symmetric, i.e.\ $J(x,y)=J(x-y)$, then the corresponding kernel $J_k$ takes the form
\[
J_k(x)=\int_{(\R^n)^{k-1}}J(x-y_1)J(y_1-y_2)\cdots J(y_{k-1})dy_{k-1}\cdots dy_1,
\]
from which it can be easily seen (via a simple change of variables) that $J_k$ is also radially symmetric. For a symmetric kernel $J$, its $k$-fold product kernels are known as the convolution $J_k=J\ast\cdots\ast J$. %For example, 
%\begin{align*}
%&J_2(x)=\int_{\R^n}J(x-y)J(y)dy,\\
%&J_3(x)=\int_{\R^n}\int_{\R^n}J(x-y)J(y-z)J(z)dzdy,
%\end{align*}
%and etc.%\medskip

\subsection*{b. Representation formula of solutions}

Next, we find a representation formula for solutions of Eqn.\ (\ref{Eqn:main}). We employ the canonical transformation 
\[
v=e^{\chi_0t}u,
\]
so that the equation becomes
\[
v(t)=u_0+\int_0^t\CJ v(\tau)d\tau\quad(t\geq0).
\]
Formally performing the Picard iteration, it follows that $v$ should satisfy
\begin{align*}
&v(t)=u_0+\int_0^t\CJ\left(u_0+\int_0^{\tau_1}\CJ v(\tau_2)d\tau_2\right)d\tau_1\\
&\hphantom{v(t)}=u_0+t\CJ u_0+\int_0^t\int_0^{\tau_1}\CJ^2v(\tau_2)d\tau_2d\tau_1,\\
&v(t)=u_0+\int_0^t\CJ\left(u_0+\tau_1\CJ u_0+\int_0^{\tau_1}\int_0^{\tau_2}\CJ^2 v(\tau_3)d\tau_3d\tau_2\right)d\tau_1\\
&\hphantom{v(t)}=
u_0+t\CJ u_0+\frac{t^2}{2!}\CJ^2u_0+\int_0^t\int_0^{\tau_1}\int_0^{\tau_2}\CJ^3v(\tau_3)d\tau_3d\tau_2d\tau_1,\\
&v(t)=u_0+t\CJ u_0+\cdots+\frac{t^k}{k!}\CJ^ku_0+\int_0^t\int_0^{\tau_1}\cdots\int_0^{\tau_{k}}\CJ^{k+1}v(\tau_{k+1})d\tau_{k+1}\cdots d\tau_1,\quad k\in\BN.
\end{align*}
Hence if $\CJ^{k+1}u_0$ decays sufficiently fast as $k\to\I$, we can conclude that $v$ must have the form
\[
v(t)=u_0+t\CJ u_0+\cdots+\frac{t^k}{k!}\CJ^ku_0+\cdots.
\]

Inverting back the above consideration, we now set the following definition.

\begin{definition}\label{Def:SolGreen}

By a solution to the nonlocal equation (\ref{Eqn:main}) with a given initial value $u_0$, we mean the function
\begin{align*}
u(t)=\CG(t)u_0:=e^{-\chi_0t}\sum_{k=0}^\I\frac{t^k}{k!}\CJ^ku_0.
\end{align*}
The operator $\CG(t)$ is the \textit{Green operator} for (\ref{Eqn:main}) whose kernel $G(x,y,t)$ is given by
\begin{align*}
G(x,y,t)=e^{-\chi_0t}\sum_{k=0}^\I\frac{t^k}{k!}J_k(x,y),
\end{align*}
where each $J_k$ is the kernel of $\CJ^k$.

\end{definition}

The following result is directly followed from the definition.

\begin{lemma}\label{Lem:JkL1}

If $J$ is a radially symmetric $L^1$ function, then $J_k\in L^1(\R^n)$ for all $k\in\BN$ and
\[
\|J_k\|_{L^1}\leq\|J\|_{L^1}^k.
\]
Moreover, if $u_0\in L^\I(\R^n)$ then
\[
\|\CG(t)u_0\|_{L^1}\leq e^{-(\chi_0-\|J\|_{L^1})t}\|u_0\|_{L^\I}.
\]

\end{lemma}

\begin{proof}
We have
\begin{align*}
\|J_k\|_{L^1}&=\int_{\R^n}\left|\int_{(\R^n)^{k-1}}J(x-y_1)J(y_1-y_2)\cdots J(y_{k-1})dy_{k-1}\cdots dy_1\right|dx\\
&\leq\int_{(\R^n)^k}|J(x-y_1)||J(y_1-y_2)|\cdots|J(y_{k-1})|dxdy_{1}\cdots dy_{k-1}=\|J\|_{L^1}^k.
\end{align*}
For the second assertion, we use Young's inequality to get
\begin{align*}
\|\CG(t)u_0\|_{L^1}\leq e^{-\chi_0t}\sum_{k=0}^\infty\frac{t^k}{k!}\|J_k\|_{L^1}\|u_0\|_{L^\infty}=e^{-(\chi_0-\|J\|_{L^1})t}\|u_0\|_{L^1}.\QED
\end{align*}
\end{proof}

\subsection*{c. Some tools from Analysis}

We will need the following facts in our study of exponential type series whose coefficients are modified by a regular varying sequence.

\begin{lemma}[\cite{TricErde51}]\label{Lem:ratiogamma}

Let $\alpha,\beta\in\R$. Then the ratio of Gamma functions has the asymptotic expansion
\[
\frac{\varGamma(s+\alpha)}{\varGamma(s+\beta)}=s^{\alpha-\beta}\left(1+\frac{(\alpha-\beta)(\alpha+\beta-1)}{2s}+O(s^{-2})\right)\quad\mbox{as $s\to\I$}.
\]
In particular, we have
\[
\frac{\varGamma(s+\alpha)}{\varGamma(s+\beta)}\leq Cs^{\alpha-\beta}\quad\mbox{as $s\to\infty$}.
\]
\end{lemma}

\begin{lemma}[\cite{BealsWong10}]\label{Lem:Kummer}

For $a,b\in\R$ with $-b\not\in\R\cup\{0\}$, Kummer's confluent hypergeometric function of the first kind is defined by
\[
M(a,b,s):=\sum_{k=0}^\infty\frac{(a)_k}{(b)_k}\frac{s^k}{k!},
\]
where $(a)_k=a(a+1)\cdots(a+k-1)$ is the Pochhammer symbol. If $b>a>0$ then
\[
M(a,b,s)\sim\frac{\varGamma(b)}{\varGamma(a)}s^{a-b}e^s\quad\mbox{as $s\to\infty$}.
\]

\end{lemma}

Next, we recall the notion of regular varying functions. 

\begin{definition}\label{Def:RegSlow}

A measurable function $R:[N_0,\infty)\to(0,\infty)$, where $N_0>0$, is called regular varying with index $\beta\in\R$ if it satisfies
\[
\lim_{s\to\infty}\frac{R(\lambda s)}{R(s)}=\lambda^\beta\quad\mbox{for all $\lambda>0$}.
\]
A slowly varying function $L$ is a regular varying function with index $\beta=0$, or, it is characterized by
\[
\lim_{s\to\infty}\frac{L(\lambda s)}{L(s)}=1\quad\mbox{for all $\lambda>0$}.
\]

\end{definition}

It is a fact that $R$ is regular varying function with index $\beta$ if and only if it can be expressed as
\[
R(s)=s^\beta L(s),
\]
where $L$ is slowly varying. For further properties see  \cite{Bingham89}.\medskip

We also need the following crucial lemma.

\begin{lemma}[\cite{Karamata30},\cite{Bingham89}]\label{Lem:Slow}

If $L$ is a slowly varying function and $\varepsilon>0$ then
\begin{align*}
&\sup_{\tau\leq s}\tau^{\varepsilon}L(\tau)\sim s^{\varepsilon}L(s),\\
&\sup_{s\geq\tau}\tau^{-\varepsilon}L(\tau)\sim s^{-\varepsilon}L(s),
\end{align*}
as $s\to\infty$. Here $f(s)\sim g(s)$ as $s\to\infty$ means $\lim_{s\to\infty}f(s)/g(s)=1$.

\end{lemma}

\begin{lemma}\label{Lem:Slow2}

If $L$ is a slowly varying function and $\varepsilon>0$ then
\begin{align*}
&\inf_{\tau\leq s}\tau^{-\varepsilon}L(\tau)\sim s^{-\varepsilon}L(s),\\
&\inf_{\tau\geq s}\tau^{\varepsilon}L(\tau)\sim s^{\varepsilon}L(s),
\end{align*}
as $s\to\infty$.

\end{lemma}

\begin{proof}
Observe that if $L$ is slowly varying then so is $K=1/L$. We have
\begin{align*}
\inf_{\tau\leq s}\tau^{-\varepsilon}L(\tau)=\frac{1}{\sup_{\tau\leq s}\tau^\varepsilon K(\tau)}\sim\frac{1}{s^\varepsilon K(s)}=s^{-\varepsilon}L(s)\quad(s\to\infty),
\end{align*}
which proves the first assertion. The second one follows by the same argument.\QED
\end{proof}

Finally, for a radially symmetric kernel $J$, in order to get the $L^\I$ bound for $k$-fold convolution kernels $J_k$ we shall need the following important result.

\begin{lemma}[Brascamp-Lieb inequality \citep{BrascampLieb76}, \citep{LiebLoss01}]\label{Lem:SharpYoung}
Let $p_1,\ldots,p_k,r\in[1,\I]$ ($k\geq2$) be such that 
\[
\frac{1}{p_1}+\cdots+\frac{1}{p_k}=k-1+\frac{1}{r}.
\]
Then
\[
\|f_1\ast\cdots\ast f_k\|_{L^r}\leq\left(\prod_{l=1}^kC_{p_l}\right)^n\|f_1\|_{L^{p_1}}\cdots\|f_k\|_{L^{p_k}}
\]
for all $f_1\in L^{p_1}(\R^n),\ldots,f_k\in L^{p_k}(\R^n)$, where $C_p$ for each $1\leq p\leq\I$ is defined by
\[
C_p=\left(\frac{p^{1/p}}{q^{1/q}}\right)^{1/2},\quad\frac{1}{p}+\frac{1}{q}=1.
\]

\end{lemma}

\section{Nonlocal equations with regular varying decay solutions}

Our first main result is a bound for an exponential type series
\begin{align}\label{Series:Rk}
\sum_{k=N}^\infty\frac{(\alpha t)^k}{k!}R_k\quad\mbox{when $R_k=k^{\B}$},
\end{align}
where $\B\in\R$. Although the result is true for all real number $\B$, the most important case in our study of nonlocal equations is when $\B<0$.

\begin{theorem}\label{Thm:Kummer}

Let $N\in\BN$ and $\A>0,\B\in\R$ be constants. Then there are $C=C(N,\beta,\alpha)>0$ and $t_0>0$ such that
\begin{align}\label{Est:beta}
\frac{1}{C}(\alpha t)^\beta e^{\alpha t}\leq\sum_{k=N}^\I\frac{(\A t)^k}{k!}k^{\B}\leq C(\A t)^{\B}e^{\A t}\quad\forall\,t\geq t_0.
\end{align}
In particular, 
\[
\sum_{k=N}^\infty\frac{(\alpha t)^k}{k!}k^\beta\asymp(\alpha t)^\beta e^{\alpha t}\quad\mbox{as $t\to\infty$}.
\]

\end{theorem}

\begin{proof}
First we prove the upper bound. By splitting the series into the sums over $N\leq k<N_0$ and that over $k\geq N_0$, where $N_0>\B$ is fixed, and noting that the first sum obviously satisfies $\lesssim\lng\alpha t\rng^\beta e^{\alpha t}$, it suffices to prove the desired upper estimate under the assumption that
\[
N>\beta.
\]

Replacing $k\to k+N$, we rewrite the series as
\begin{align*}
\sum_{k=N}^\I\frac{(\A t)^k}{k!}k^{\B}&=(\A t)^N\sum_{k=0}^\I\frac{k!}{(k+N)!(k+N)^{-\B}}\frac{(\A t)^k}{k!}.
\end{align*}

\begin{claim}
There is a constant $C_1=C_1(N,\beta)>0$ such that
\[
(k+N)!(k+N)^{-\B}\geq C_1\prod_{l=1}^k\left(l+N-\B\right)
\quad\forall\,k\geq0.
\]

\end{claim}

\begin{proof}
The desired estimate is equivalent to that
\begin{align*}
(k+N)^{-\B}\geq C_2\frac{\varGamma(k+N-\B+1)}{\varGamma(k+N+1)},
\end{align*}
where $C_2:=C_1/\varGamma(N-\B+1)$, for some constant $C_2$. By the work of Tricomi and Erd\'elyi on asymptotic ratio of Gamma functions (Lemma \ref{Lem:ratiogamma}) we have
\[
\frac{\varGamma(k+N-\B+1)}{\varGamma(k+N+1)}\sim(k+N)^{-\B}\quad\mbox{as $k\to\I$},
\]
so there is a constant $C_3=C_3(N,\beta)>0$ such that
\[
\frac{\varGamma(k+N-\B+1)}{\varGamma(k+N+1)}\leq C_3(k+N)^{-\B}\quad\forall\,k\geq0.
\]
Taking $C_2=1/C_3$, thus there is such a constant $C_2$ as claimed.\QED
\end{proof}

According to the claim, we now have
\begin{align*}
\sum_{k=N}^\I\frac{(\A t)^k}{k!}k^{\B}&=(\A t)^N\sum_{k=0}^\I\frac{k!}{(k+N)!(k+N)^{-\B}}\frac{(\A t)^k}{k!}\\
&\leq C_{N,\beta}(\A t)^N\sum_{k=0}^\I\frac{k!}{(1+N-\B)(2+N-\B)\cdots(k+N-\B)}\frac{(\A t)^k}{k!}\\
&=C_{N,\beta}(\A t)^N\sum_{k=0}^\I\frac{(1)_k}{(1+N-\B)_k}\frac{(\A t)^k}{k!},
\end{align*}
where $(a)_k:=a(a+1)\cdots(a+k-1)$ denotes the Pochhammer symbol.\medskip

The last series on the right hand side above takes the form of \textit{Kummer's confluent hypergeometric function of the first kind} (Lemma \ref{Lem:Kummer}). As $t\to\I$, we have that
\[
\sum_{k=0}^\I\frac{(1)_k}{(1+N-\B)_k}\frac{(\A t)^k}{k!}\sim\frac{\varGamma(1+N-\B)}{\varGamma(1)}(\A t)^{-(N-\B)}e^{\A t}.
\]
So we can choose $t_0>0$ such that
\begin{align*}
\sum_{k=N}^\I\frac{(\A t)^k}{k!}k^{\B}\leq C_{N,\beta,\alpha}(\alpha t)^N(\A t)^{-N+\B}e^{\A t}=(\alpha t)^\beta e^{\alpha t}\quad\mbox{for all $t\geq t_0$},
\end{align*}
which implies the desired upper estimate.\medskip

Next we prove the lower estimate of (\ref{Est:beta}). We can consider $t\geq1$. Also, to derive the lower bound, we can take $N>\beta$. According to the prove of the claim above, there is a constant $\tilde{C}_1=\tilde{C}_1(N,\beta)>0$ such that
\begin{align*}
(k+N)!(k+N)^{-\beta}%&=(k+N)^{-\beta}\left\{\frac{1}{\varGamma(N-\beta+1)}\frac{\varGamma(k+N-\beta+1)}{\varGamma(k+N+1)}\right\}\\
&\leq \tilde{C}_1\prod_{l=1}^k\left(l+N-\B\right)\quad\forall\,k\geq0.
\end{align*}
Then
\begin{align*}
\sum_{k=N}^\infty\frac{(\alpha t)^k}{k!}k^\beta&\gtrsim(\alpha t)^N\sum_{k=0}^\infty\frac{k!}{(1+N-\beta)(2+N-\beta)\cdots(k+N-\beta)}\frac{(\alpha t)^k}{k!}\\
&=(\alpha t)^N\sum_{k=0}^\infty\frac{(1)_k}{(1+N-\beta)_k}\frac{(\alpha t)^k}{k!}
\end{align*}
and so we conclude by the asymptotic behavior of Kummer's confluent hypergeometric function of the first kind once again, we find that
\[
\sum_{k=N}^\infty\frac{(\alpha t)^k}{k!}k^\beta\gtrsim(\alpha t)^\beta e^{\alpha t}
\]
as needed. \QED

\end{proof}

\begin{remark}

In the special case that $\alpha=1$, $N=0$, and $\beta=n$ is a positive integer, the summation
\[
B_n(t)=e^{-t}\sum_{k=0}^\infty\frac{t^k}{k!}k^n
\]
is the Bell polynomial, which is a polynomial in $t$ of degree $n$.

\end{remark}

Now we study the case that the exponential type series (\ref{Series:Rk}) has
\[
R_k=R(k)\quad\mbox{where $R(s)$ is a regular varying function}.
\]
See Definition \ref{Def:RegSlow} and Lemma \ref{Lem:Slow}.\medskip

The following result is a generalization of Theorem \ref{Thm:Kummer}, though, its proof relies crucially on the result of the preceding theorem.

\begin{theorem}\label{Thm:Rk}

Let $N\in\BN$, $\alpha>0$, and $R$ be a regular varying function with index $\beta\in\R$. Let $R_k=R(k)$ for any positive integer $k$. Then there are constants $C,t_0>0$ such that
\begin{align}\label{Est:Rk}
\frac{1}{C}R(\alpha t)e^{\alpha t}\leq\sum_{k=N}^\infty\frac{(\alpha t)^k}{k!}R_k\leq CR(\alpha t)e^{\alpha t}\quad\mbox{for all $t\geq t_0$}.
\end{align}
In particular,
\[
\sum_{k=N}^\infty\frac{(\alpha t)^k}{k!}R_k\asymp R(\alpha t)e^{\alpha t}\quad\mbox{as $t\to\infty$}.
\]

\end{theorem}

\begin{proof}
Let us split the series into
\begin{align*}
\sum_{N\leq k<t}\frac{(\alpha t)^k}{k!}R_k+\sum_{k\geq t}\frac{(\alpha t)^k}{k!}R_k=:\CS_{1}+\CS_{2}.
\end{align*}
We will use Lemma \ref{Lem:Slow} to prove the upper bound. Let $R(s)=s^{\beta}L(s)$ where $L$ is a slowly varying function and $\beta\in\R$. Take $\varepsilon>0$. Then we have
\begin{align*}
\CS_1&=\sum_{N\leq k<t}\frac{(\alpha t)^k}{k!}k^{\beta-\varepsilon}\cdot(k^{\varepsilon}L(k))\\
&\leq\sup_{N\leq k\leq t}k^{\varepsilon}L(k)\sum_{N\leq k<t}\frac{(\alpha t)^k}{k!}k^{\beta-\varepsilon}.
\end{align*}
By Lemma \ref{Lem:Slow} we have that
\begin{align*}
\sup_{N\leq k\leq t}k^{\varepsilon}L(k)\sim t^\varepsilon L(t)\quad\mbox{as $t\to\infty$},
\end{align*}
so we get by Theorem \ref{Thm:Kummer} that
\begin{align*}
\CS_1&\lesssim t^{\varepsilon}L(t)\sum_{N\leq k<t}\frac{(\alpha t)^k}{k!}k^{\beta-\varepsilon}\lesssim t^{\varepsilon}L(t)(\alpha t)^{\beta-\varepsilon}e^{\alpha t}\\
%&=\alpha^{-\varepsilon}(\alpha t)^\beta L(t)e^{\alpha t}\\
&\lesssim\alpha^{-\varepsilon}(\alpha t)^\beta L(\alpha t)e^{\alpha t}=\alpha^{-\varepsilon}R(\alpha t)e^{\alpha t},
\end{align*}
as $t\to\infty$. Here, in the last inequality, we have used that $L$ is slowly varying, hence
\[
L(\alpha t)\sim L(t)\quad\mbox{as $t\to\infty$}.
\]
Note that we may take $\varepsilon=\alpha$ so that $\alpha^{-\varepsilon}$ is bounded by a constant independent of $\alpha>0$.\medskip

Next we establish the upper bound of $\CS_2$. As $t\to\infty$, we have by Lemma \ref{Lem:Slow} that
\[
\sup_{k\geq t}k^{-\varepsilon}L(k)\sim t^{-\varepsilon}L(t),
\]
hence there are constants $C,t_0>0$ such that
\[
\sup_{k\geq t}k^{-\varepsilon}L(k)\leq Ct^{-\varepsilon}L(t)\quad\mbox{for all $t\geq t_0$}.
\]
Then we have, for $t\geq t_0$, that
\begin{align*}
\CS_2&=\sum_{k\geq t}\frac{(\alpha t)^k}{k!}k^{\beta+\varepsilon}\cdot(k^{-\varepsilon}L(k))\leq\sup_{k\geq t}k^{-\varepsilon}L(k)\sum_{k\geq t}\frac{(\alpha t)^k}{k!}k^{\beta+\varepsilon}\\
&\leq Ct^{-\varepsilon}L(t)\sum_{k\geq t}\frac{(\alpha t)^k}{k!}k^{\beta+\varepsilon}\leq Ct^{-\varepsilon}L(t)(\alpha t)^{\beta+\varepsilon}e^{\alpha t}\\
&\leq C\alpha^{\varepsilon}(\alpha t)^\beta L(\alpha t)e^{\alpha t}=C\alpha^\varepsilon R(\alpha t)e^{\alpha t}.
\end{align*}
Combining the upper estimates for $\CS_1,\CS_2$, we obtain the upper estimate in (\ref{Est:Rk}).\medskip

It remains to prove the lower estimate in (\ref{Est:Rk}). Clearly,
\[
\sum_{k=N}^\infty\frac{(\alpha t)^k}{k!}R_k\geq\CS_1,
\]
so it suffices to establish the lower bound of $\CS_1$. We use Lemma \ref{Lem:Slow2}. Consider
\begin{align*}
\CS_1&=\sum_{N\leq k<t}\frac{(\alpha t)^k}{k!}k^{\beta+\varepsilon}\cdot(k^{-\varepsilon}L(k))\\
&\geq\inf_{N\leq k\leq t}k^{-\varepsilon}L(k)\sum_{N\leq k<t}\frac{(\alpha t)^k}{k!}k^{\beta+\varepsilon}\\
&\gtrsim t^{-\varepsilon}L(t)\sum_{N\leq k<t}\frac{(\alpha t)^k}{k!}k^{\beta+\varepsilon}\quad(\mbox{by Lemma \ref{Lem:Slow2}})\\
&\gtrsim t^{-\varepsilon}L(t)(\alpha t)^{\beta+\varepsilon}e^{\alpha t},
\end{align*}
where, in the last inequality, we have used that
\[
\sum_{N\leq k<t}\frac{(\alpha t)^k}{k!}k^{\beta+\varepsilon}\sim\sum_{k=N}^\infty\frac{(\alpha t)^k}{k!}k^{\beta+\varepsilon}\gtrsim(\alpha t)^{\beta+\varepsilon}e^{\alpha t}\quad\mbox{as $t\to\infty$, by Theorem \ref{Thm:Kummer}}.
\]
Finally, using the slow variation of $L$, we obtain the desired lower estimate for $\CS_1$.\QED

\end{proof}

\begin{remark}

For the case that $R$ is regular varying with index $\beta\leq0$, the preceding result was established in \cite{BinghamEtal83} using a probabilistic argument. Here we use analytic argument and obtain the case $\beta>0$ as well.

\end{remark}

\begin{example} There are many regular varying sequences. So by applying the preceding theorem, we get the following interesting conclusion.
\begin{itemize}
\item[(1)] For $L(s)=(\ln k)^\mu$ which clearly satisfies $\lim_{s\to\infty}L(\lambda s)/L(s)=1$ for all $\lambda>0$, we have
\begin{align}
\sum_{k=N}^\infty\frac{(\alpha t)^k}{k!}k^\beta(\ln k)^\mu\asymp (\alpha t)^\beta(\ln(\alpha t))^\mu e^{\alpha t}\quad\mbox{as $t\to\infty$},
\end{align}
where $N\in\BN$, $\alpha>0$, and $\beta,\mu\in\R$.
\item[(2)] More generally, $L(s)=(\ln s)^{\mu_1}\cdots(\ln_ms)^{\mu_m}$ ($\ln_j=\ln\circ\cdots\circ\ln$, $j$ terms) is slowly varying, so we get
\begin{align}
\sum_{k=N}^\infty\frac{(\alpha t)^k}{k!}k^\beta(\ln k)^{\mu_1}\cdots(\ln_mk)^{\mu_m}\asymp(\alpha t)^\beta(\ln (\alpha t))^{\mu_1}\cdots(\ln_m(\alpha t))^{\mu_m}\quad\mbox{as $t\to\infty$},
\end{align}
where $N\in\BN,\alpha>0$, and $\beta,\mu_1,\ldots,\mu_m\in\R$.
\item[(3)] (Non-logarithmic slowly varying function). One can show that (see \cite{Bingham89})
\[
L(s)=\exp\left((\ln s)^{\mu_1}\cdots(\ln_ms)^{\mu_m}\right),
\]
is slowly varying for any $\mu_1,\ldots,\mu_m\in\R$. So we get
\begin{align*}
\sum_{k=N}^\infty\frac{(\alpha t)^k}{k!}k^\beta\exp\left((\ln k)^{\mu_1}\cdots(\ln_mk)^{\mu_m}\right)\asymp (\alpha t)^\beta\exp\left((\ln(\alpha t))^{\mu_1}\cdots(\ln_m(\alpha t))^{\mu_m}\right)\quad\mbox{as $t\to\infty$},
\end{align*}
for all $N\in\BN,\alpha>0$, and $\beta,\mu_1,\ldots,\mu_m\in\R$.
\item[(4)] One also has oscillating slowly varying function (see \cite{Bingham89})
\[
L(s)=\exp\left((\ln s)^{1/3}(\cos(\ln s))^{1/3}\right).
\]
\item[(5)] By Karamata's representation theorem, $L$ is slowly varying if and only if there are measurable functions $c(s),\varepsilon(s)$ such that $\lim_{s\to\infty}c(s)=c_0>0$ and $\lim_{s\to\infty}\varepsilon(s)=0$ such that
\[
L(s)=c(s)\exp\left(\int_{s_0}^s\frac{\varepsilon(\tau)}{\tau}d\tau\right).
\]
\end{itemize}

\end{example}

\begin{remark}
Using the result of Theorem \ref{Thm:Rk}, we will be able to give examples of nonlocal diffusion equations having arbitrarily regular varying decay solutions.
\end{remark}

We conclude this section with the following sufficient condition on the kernel $J$ of (\ref{Eqn:main}) such that the decay of solutions is at the rate of a regular varying function. The result is true not only for radially symmetric nonlocal equations but also for non-symmetric ones. Examples of equations satisfying the hypothesis of this theorem will be presented in later section. For simplicity of the presentation, we consider $\chi_0=1$ in (\ref{Eqn:main}).

\begin{theorem}\label{Thm:GenDecay}

Let $\chi_0=1$ in (\ref{Eqn:main}) and let $R(t)=t^{\beta}L(t)$ be a regular varying function with index $\beta\in\R$. Assume there is a positive integer $N\in\BN$ such that either (i) $u_0\in L^1(\R^n)\cap L^\infty(\R^n)$ and
\begin{align}\label{Hyp:H2}\tag{H1}
\begin{cases}
\displaystyle
\sup_{x\in\R^n}\int_{\R^n}|J(x,y)|dy<\infty,
\\
\vspace{-7pt}\\
\displaystyle
|J_k(x,y)|\leq R_k\quad\forall\,x,y\in\R^n, k=N,N+1,\ldots,
\end{cases}
\end{align}
where $R_k:=R(k)$, or (ii) there is $1\leq p\leq\infty$ such that
\begin{align}\label{Hyp:H3}\tag{H2}
\begin{cases}
\displaystyle
\CJ_ku_0\in L^p(\R^n)&k=1,2,\ldots,N-1,
\\
\vspace{-7pt}\\
\displaystyle
\|\CJ^ku_0\|_{L^p}\leq R_k&k=N,N+1,\ldots
\end{cases}
\end{align}
Then the solution $u(t)$ of (\ref{Eqn:main}) satisfies
\begin{align*}
\|u(t)\|_{L^q}\lesssim t^{\beta}L(t)\quad\mbox{as $t\to\infty$},
\end{align*}
where $q=\infty$ in case (i) and $q=p$ in case (ii).

\end{theorem}

\begin{proof}
We split the solution into
\begin{align*}
u(t)&=\CG(t)u_0=e^{-t}\sum_{0\leq k<N}\frac{t^k}{k!}\CJ^ku_0+e^{-t}\sum_{k\geq N}\frac{t^k}{k!}\CJ^ku_0=:\CS_1+\CS_2.
\end{align*}
See Definition \ref{Def:SolGreen}.\medskip

Assume (i). By the first part of (\ref{Hyp:H2}) and the Fubini's theorem, we get that
\[
\int_{\R^n}|J_k(x,y)|dy\leq\int_{\R^n}\int_{(\R^n)^{k-1}}|J(x,y_1)J(y_1,y_2)\cdots J(y_{k-1},y)|dy_{k-1}\cdots dy_1dy\leq M^k
\]
where $M=\sup_{x}\int_{\R^n}|J(x,y)|dy<\infty$. The first term $\CS_1$ can be estimated by
\begin{align*}
|\CS_1|&\leq e^{-t}\sum_{0\leq k<N}\frac{t^k}{k!}\int_{\R^n}|J_k(x,y)u_0(y)|dy\\
&\leq e^{-t}\|u_0\|_{L^\infty}\left\{\max_{1\leq k\leq N-1}M^k\right\}\sum_{0\leq k<N}\frac{t^k}{k!}\\
&\leq C_{M,N}\|u_0\|_{L^\infty}e^{-t}\sum_{0\leq k<N}\frac{t^k}{k!}\\
&\leq C_{M,N}\|u_0\|_{L^\infty}t^{-|\beta|}L(t),
\end{align*}
as $t\to\infty$. For the second term $\CS_2$, we use the second part of the hypothesis (\ref{Hyp:H2}) and apply Theorem \ref{Thm:Rk} to get that
\begin{align*}
|\CS_2|&\leq e^{-t}\sum_{k\geq N}\frac{t^k}{k!}\int_{\R^n}|J_k(x,y)u_0(y)|dy\\
&\leq e^{-t}\|u_0\|_{L^1}\sum_{k\geq N}\frac{t^k}{k!}R_k\\
&\lesssim e^{-t}\|u_0\|_{L^1}R(t)e^{t}=\|u_0\|_{L^1}t^{\beta}L(t),
\end{align*}
as $t\to\infty$. Combining the inequalities of $\CS_1,\CS_2$, we get the desired estimate.\medskip

Now assume (ii). Again, we split $u(t)=\CS_1+\CS_2$. For $\CS_1$, we apply the triangle inequality to get
\begin{align*}
\|\CS_1\|_{L^p}&\leq e^{-t}\sum_{0\leq k<N}\frac{t^k}{k!}\|\CJ_ku_0\|_{L^p}\\
&\leq\left(\max_{1\leq k\leq N_1}\|\CJ_ku_0\|_{L^p}\right)e^{-t}\sum_{0\leq k<N}\frac{t^k}{k!}\\
&\leq Ct^\beta L(t)
\end{align*}
as $t\to\infty$. For the second term, we use
\begin{align*}
\|\CS_2\|_{L^p}&\leq e^{-t}\sum_{k\geq N}\frac{t^k}{k!}\|\CJ_ku_0\|_{L^p}\\
&\leq e^{-t}\sum_{k\geq N}\frac{t^k}{k!}R_k\\
&\leq Ce^{-t}R(t)e^t=Ct^\beta L(t)
\end{align*}
by Theorem \ref{Thm:Rk}. Combining the both estimates, we conclude the assertion for (ii).\QED

\end{proof}

\section{Integral operators with higher integrability}

In this section we apply results from the previous section to study the asymptotic behavior of solutions to the nonlocal equation (\ref{Eqn:main}) if the kernel $J$ is a radially symmetric $L^1$ function:
\begin{align}\label{Hyp:alpha2}\tag{H3}
J=J(x-y),\quad\int_{\R^n}|J(x)|dx=1.
\end{align}
Although the decay estimate derived in this section is now a classical result, our way of getting the estimate provides an alternative point of view. Additionally, the bound of kernels $J_k$ obtained (Proposition \ref{Prop:EstJk}) is new and could be useful in other discipline.\medskip

Before discussing our next main result, let us briefly recall the following basic fact.

\begin{lemma}[\cite{Feller71},\cite{Caravenna12}]

Let $f\in L^1(\R^n)$ and $f_k:=f\ast\cdots\ast f$ denote the $k$-fold convolution of $f$
\begin{itemize}
\item[(i)] $f_k\in L^\infty(\R^n)$ for some $k$ if and only if $\widehat{f}\in L^q(\R^n)$ for some $1\leq q<\infty$.
\item[(ii)] If $f_N\in L^\infty(\R^n)$ for some $N$ then $f_k\in L^\infty(\R^n)$ for all $k\geq N$.
\item[(iii)] If $f\in L^{1+\varepsilon_0}(\R^n)$ for some $\varepsilon_0>0$, then $f_k\in L^\infty(\R^n)$ for all $k$ large enough.
\end{itemize}

\end{lemma}

\begin{proof}
(ii) is obvious. The proof of (i) and (iii) can be found in \cite{Caravenna12}, \citep{Feller71}. We will present a more precise assertion (Proposition \ref{Prop:EstJk}) than (iii) which also provides the bound of the sup-norm $\|f_k\|_{L^\infty}\sim k^{-n/2}$. \QED
\end{proof}

Now we present our next main result. The proof uses the Brascamp-Lieb inequality (or sharp Young's convolution inequality), Lemma \ref{Lem:SharpYoung}. We note that $J$ can be real, or complex valued.

\begin{proposition}\label{Prop:EstJk}
Assume (\ref{Hyp:alpha2}) and furthermore
\begin{align}\label{Hyp:J1epsilon}
J\in L^1(\R^n)\cap L^{1+\V_0}(\R^n)\quad\mbox{for some $0<\V_0\leq\infty$}.
\end{align}
Let $N=\lceil\frac{1}{\V_0}\rceil+1$. If $k\geq N$ then $J_k\in L^\I(\R^n)\cap C(\R^n)$ and there are constants $C_n,\gamma>0$ such that
\begin{align}\label{IntegralCond}
\|J_k\|_{L^\I}\leq C_{n}\exp\left(\gamma\int_{\R^n}|J(x)|\ln|J(x)|dx\right)k^{-n/2}\quad\forall\,k\geq N.
\end{align}

\end{proposition}

\begin{proof}
Without loss of generality, we can assume $J\geq0$. Clearly, $J\in L^p(\R^n)$ for all $1\leq p\leq1+\V_0$ by interpolation. We apply the Brascamp-Lieb (or sharp Young) inequalities (see Lemma \ref{Lem:SharpYoung}). Take $k\geq N=\lceil\frac{1}{\V_0}\rceil+1$, and put
\[
p_1=\cdots=p_k=\frac{k}{k-1},\quad r=\I,\quad f_1=\cdots=f_k=J
\]
in the Brascamp-Lieb inequality. Note that $f_l=J\in L^{p_l}(\R^n)$ for all $l$. For each $p_l$, the H\"older conjugate is $q_l=k$. We calculate
\begin{align*}
C_{p_l}^k&=\left(\frac{(k/(k-1))^{(k-1)/k}}{k^{1/k}}\right)^{k/2}%=\left(\frac{k^{k-2}}{(k-1)^{k-1}}\right)^{1/2},
\\
&=\frac{1}{k^{1/2}}\left(1+\frac{1}{k-1}\right)^{(k-1)/2}\\
&\leq\frac{\sqrt{e}}{k^{1/2}}.
\end{align*}
Now we have $J_k(x)=f_1\ast\cdots\ast f_k$. So we obtain by Lemma \ref{Lem:SharpYoung} and the above calculations that
\begin{align*}
\sup_{x\in\R^n}|J_k(x)|&\leq\frac{e^{n/2}}{k^{n/2}}\left(\int_{\R^n}J(x)^{k/(k-1)}dx\right)^{k-1}.
\end{align*}
It is obvious that $J_k$ are continuous. Thus $J_k\in BC(\R^n)$. \medskip

Consider the preceding integral as $k\to\I$. We apply the L'Hopital's rule and the dominated convergence theorem to get
\begin{align*}
\lim_{k\to\I}(k-1)\ln\int_{\R^n} J(x)^{k/(k-1)}dx&=\lim_{\L\to0^+}\frac{1}{\L}\ln\int_{\R^n} J(x)^{\L+1}dx\quad(\L:=\frac{1}{k-1}),\\
&=\int_{\R^n}J(x)\ln J(x)dx<\I.
\end{align*}
So there is a constant $\gamma>0$ such that the estimate
\[
\sup_x|J_k(x)|\leq C_n\exp\left(\gamma\int J(x)\ln J(x)dx\right)k^{-n/2}
\]
is true for all $k\geq N$.\QED
\end{proof}

\begin{remark}
\begin{itemize}
\item[(1)] In the preceding proposition, the kernels $J_k$ needs not be bounded when $k$ is small. For instance, the Bessel potential operator 
\[
\CB=(1-\Lap)^{-1}\quad\mbox{on $\R^n$ ($n>2$)}
\]
is known to have the kernel $B\in L^1(\R^n)\cap L^{1+\V_0}(\R^n)$ for any $\V_0<\frac{2}{n-2}$ and the $k$-fold iterated kernel
\[
B_k\not\in L^\I(\R^n)\quad\mbox{for $k<\frac{n}{2}$},\quad B_k\in L^\I(\R^n)\quad\mbox{for all $k\geq\frac{n}{2}$}.
\]

More generally, if $\CK$ is a weakly singular integral operator, i.e.\ its kernel $K$ satisfies
\[
|K(x)|\sim\frac{1}{|x|^{n-\alpha}}\quad\mbox{as $|x|\to0$}\quad(0<\alpha<n),
\]
$K$ is finite outside the diagonal, and $K$ decays sufficiently fast at infinity, then
\[
K_l\not\in L^\I(\R^n)\quad\mbox{for $l<\frac{n}{\alpha}$},\quad K_l\in L^\I(\R^n)\quad\mbox{for $l\geq\frac{n}{\alpha}$}.
\]
The Bessel potential $\CB$ is a weakly singular operator having $\alpha=2$ and exponential decay at infinity.
\item[(2)] A more precise result compared to Proposition \ref{Prop:EstJk} was derived in \citep{KonMolVain17} (Lemma 5.4), using the \textit{local limit theorem}, but for a rather restricted class of kernel functions (see Eq.\ (32) and (33) in \citep{KonMolVain17}). More precisely, in order to apply the local limit theorem, it was assumed in \citep{KonMolVain17} that the kernel $J$ has \textit{ultra light tail}, i.e.\
\[
|J(x)|,\,\,|\N J(x)|\lesssim e^{-|x|^\alpha}\quad\mbox{where $\alpha>1$}.
\]
It should be observed that the estimate in Proposition \ref{Prop:EstJk} is uniform, whereas, Lemma 5.4 in \citep{KonMolVain17} is true only when $|x|\leq k$. More importantly, the estimate (48) derived in \citep{KonMolVain17} seems weaker in some cases than what we have shown here.
\item[(3)] It should be noted that there are kernel functions
which do not satisfy the criterion of Proposition \ref{Prop:EstJk}, that is there are $J\in L^1(\R^n)$ such that 
\[
\|J\|_{L^{1+\V}}=\I\quad\mbox{for all $\V>0$}.
\]
A basic example is 
\[
J(x)=\frac{1}{|x|^n\{1+(\ln|x|)^2\}}
\]
for which $\|J\|_{L^1}<\I$ but $\|J\|_{L^{1+\V}}=\I$ for all $\V>0$.
\end{itemize}

\end{remark}

Using Theorem \ref{Thm:GenDecay} and Proposition \ref{Prop:EstJk}, we obtain the following decay property of solutions to (\ref{Eqn:main}).

\begin{theorem}\label{Thm:GreenJsigma}

Assume $J$ satisfies (\ref{Hyp:alpha2}) and (\ref{Hyp:J1epsilon}). If $u_0\in L^1(\R^n)\cap L^\I(\R^n)$, then the solution $u(t)$ of (\ref{Eqn:main}) satisfies
\begin{align}\label{Est:GJsigma1}
\|u(t)\|_{L^\I}\leq Ct^{-n/2}\quad\forall\,t\geq t_0.
\end{align}
Moreover,  for each $1\leq q\leq\I$, there is a constant $C>0$ independent of $q$ such that
\begin{align}\label{Est:GJsigma3}
\|u(t)\|_{L^q}\leq Ct^{-\frac{n}{2}(1-\frac{1}{q})}\quad\forall\,t\geq t_0.
\end{align}

\end{theorem}

\begin{proof}
By Proposition \ref{Prop:EstJk}, we have
\[
\|J_k\|_{L^\infty}\leq R_k:=Ck^{-n/2}\quad C=C(n,J)>0
\]
for all $k\geq N=\lceil 1/\varepsilon_0\rceil+1$. It is now clear that $J$ satisfies (\ref{Hyp:H2}) in Theorem \ref{Thm:GenDecay}. Thus we obtain
\[
\|u(t)\|_{L^\infty}\leq Ct^{-n/2}\quad\forall\,t\geq t_0>0.
\]
For (\ref{Est:GJsigma3}), we simply apply the interpolation
\[
\|u(t)\|_{L^q}\leq\|u(t)\|_{L^\I}^{1-\frac{1}{q}}\|u(t)\|_{L^1}^{\frac{1}{q}}
\]
and Lemma \ref{Lem:JkL1}.\QED

\end{proof}

If $J\in L^1(\R^n)\cap L^\infty(\R^n)$, e.g.\ $J$ is continuous with compact support, then the result in the preceding theorem can be strengthen. In this case, we have for any $u_0\in L^1(\R^n)$ (not necessarily in $L^\infty(\R^n)$) then the solution of (\ref{Eqn:main}) satisfies
\[
\|u(t)-e^{-t}u_0\|_{L^\infty}\lesssim t^{-n/2}\quad\mbox{as $t\to\infty$}.
\]
This refined asymptotic behavior was obtained in \cite{IgnatEtal08}.\medskip

Moreover, we have the following refined asymptotic behavior.

\begin{corollary}

Assume $J$ satisfies (\ref{Hyp:alpha2}) and (\ref{Hyp:J1epsilon}). Then for any $u_0\in L^1(\R^n)$, the solution of (\ref{Eqn:main}) satisfies
\[
\left\|u(t)-e^{-t}\sum_{k=0}^{N-1}\frac{t^k}{k!}\CJ^ku_0\right\|_{L^\infty}\lesssim t^{-n/2}\quad\mbox{as $t\to\infty$},
\]
where $N=\lceil\frac{1}{\varepsilon_0}\rceil+1$.

\end{corollary}

\begin{remark}\label{Rem:Stable}

In \citep{ChasseigneEtal06}, the nonlocal problem (\ref{Eqn:main}) was investigated with $\chi_0=\|J\|_{L^1}=1$. Assuming the kernel $J$ has the Fourier transform expansion
\begin{align}\label{Fourier:J1}
\widehat{J}(\xi)=1-A|\xi|^\S+o(|\xi|^\S)
\quad\mbox{as $\xi\to0$},
\end{align}
and the initial function $u_0,\widehat{u}_0\in L^1(\R^n)$, the authors were able to prove the decay estimate
\[
\|\CG(t)u_0\|_{L^\I}\leq Ct^{-n/\S}.
\]
This kind of kernel functions arises in the context of \textit{stable laws} with index $\sigma$ and it was remarked in \citep{Feller71} (the remark after Theorem 2, XV.\ 5) that if $J$ is a stable law with index $0<\sigma<2$, then necessarily 
\[
\|J_k\|_{L^\I}=\I\quad\mbox{for all $k$},
\]
since otherwise, the pointwise bound should be $t^{-n/2}$ (the normal distributions). We will address the issue that $\|J_k\|_{L^\I}=\I$ for all $k$ in Section \ref{Subsec:PtStable}. \medskip

As is noted in \cite{Alfaro17}, the expansion (\ref{Fourier:J1}) holds true for $J\geq0$, bounded, radially function $\|J\|_{L^1}$ where the algebraic tail is
\begin{align*}
J(x)\sim\frac{1}{|x|^{n+2+\varepsilon}}\quad\varepsilon>0\,\,(\sigma=2),\quad J(x)\sim\frac{1}{|x|^\alpha}\quad n<\alpha<n+2\,\,(\sigma=\alpha-n\in(0,2)),
\end{align*}
as $|x|\to\infty$. For the second type of tails, the second momentum is $\infty$. At the critical algebraic tail
\[
J(x)\sim\frac{1}{|x|^{n+2}}\quad\mbox{as $|x|\to\infty$},
\]
it follows that the Fourier expansion of $J$ behaves as
\[
\widehat{J}(\xi)=1-A|\xi|^2\ln(1/|\xi|)+o(|\xi|^2\ln(1/|x|))\quad\mbox{as $|\xi|\to0$}.
\]

\end{remark}

\begin{remark}

It is not difficult, by chasing through the proofs of results in this section that, if 
\[
\chi_0=\|J\|_{L^1},
\]
then the same decay of solution is still true under the hypotheses (\ref{Hyp:alpha2}) and (\ref{Hyp:J1epsilon}). On the other hand, if $\chi_0>\|J\|_{L^1}$ the solution shows exponential decay instead.
This partially answers an open question raised in \citep{ChasseigneEtal06} about the asymptotic behavior of solutions to nonlocal diffusion equations when $\chi_0=1,\|J\|_{L^1}\neq1$.
\end{remark}

\begin{remark}

In \citep{ChasseigneEtal14}, a fractional decay estimate was proved similar to our result. Under the assumption that $J\in C(\R^n)$ and $J\gtrsim|x-y|^{-(n+2\S)}$ for $|x-y|$ large, they obtained the asymptotic behavior
\[
\|u(t)\|_{L^q}\leq C_qt^{-\frac{n}{2\S}(1-\frac{1}{q})}
\]
as $t\to\I$. This estimate is true for $0<\S<1$ and $1\leq q<\I$. More importantly, the constant $C_q$ depends on $q$. See also \citep{IgnatRossi09}.

\end{remark}

\section{Integral operators of stable laws}\label{Subsec:PtStable}

In this section, we study the nonlocal equation (\ref{Eqn:main}) when the kernel $J\geq0$ is radially symmetric and represents a stable law with index $0<\S<2$. For simplicity, we assume $\chi_0=\|J\|_{L^1}=1$. The decay estimates of solutions obtained in \cite{ChasseigneEtal06} will be reproved and generalized from the power series point of view (see Corollary \ref{Cor:DecaySigma}, and Theorem \ref{Thm:DecayLn} for the logarithmic perturbation case). The results will be further generalized to discover the decay of solutions of (\ref{Eqn:main}) satisfying a regular varying function in the next section.\medskip

As in Remark \ref{Rem:Stable}, for stable laws, $\|J_k\|_{L^\I}=\I$ for all $k$, so Proposition \ref{Prop:EstJk} is useless. To deal with this situation, it is necessarily to impose a certain integrability property for the initial condition $u_0$
so that $\CJ_ku_0\in L^p(\R^n)$ ($1\leq p\leq\I$) for all $k$. This is essentially the key idea in deriving the decay estimate (but with $p=\infty$) in \cite{ChasseigneEtal06} via the Fourier splitting technique. 
We show that $\|\CJ_ku_0\|_{L^p}$ is regular varying (w.r.t.\ $k$) and the decay estimate then follows directly from Theorem \ref{Thm:Rk}.
\medskip

Let us state the assumption for the following theorem:
\begin{align}\tag{H4}\label{Fourier:Jsigma}
\begin{cases}
\displaystyle
J=J(|x|)\geq0,\quad\chi_0=\|J\|_{L^1}=1,\\
\vspace{-10pt}\\
\displaystyle
\widehat{J}(\xi)=1-A|\xi|^\sigma+o(|\xi|^\sigma)\,\,\mbox{as $\xi\to0$},\\
\vspace{-10pt}\\
\displaystyle
\hspace{2cm}\mbox{where $0<\sigma\leq2,A>0$,}\\
\vspace{-10pt}\\
\displaystyle
u_0\in L^1(\R^n),\,\,\widehat{u}_0\in L^1(\R^n).
\end{cases}
\end{align}
Note that, the assumption $\widehat{u}_0\in L^1(\R^n)$ implies $u_0\in L^\infty(\R^n)$ by the Fourier inversion formula.
Such a kernel $J$ is also known as dispersal kernel \cite{Alfaro17}.

\begin{theorem}\label{Thm:StableJk}

Assume (\ref{Fourier:Jsigma}). Let $1\leq p\leq\infty$. Then, 
\[
\CJ_ku_0\in L^{p}(\R^n)\quad\mbox{for all $k$},
\]
and, furthermore, there is a positive integer $N=N(n,J)$ such that
\begin{align}\label{Est1:ThmStableJk}
\|\CJ_ku_0\|_{L^{p}}\leq C(\|u_0\|_{L^1}+\|\widehat{u}_0\|_{L^1})k^{-\frac{n}{\sigma}(1-\frac{1}{p})}\quad\forall\,k\geq N,
\end{align}
where $C>0$ is a constant depending only on $n,J$.

\end{theorem}

\begin{proof}
It suffices to establish the case $p=\infty$. In fact, after doing so, we simply apply the interpolation inequality
\[
\|\phi\|_{L^p}\leq\|\phi\|_{L^1}^{1/p}\|\phi\|_{L^{\infty}}^{1-1/p},
\]
together with the preservation of $L^1$ norms (Lemma \ref{Lem:JkL1}): $\|\CJ_ku_0\|_{L^1}\leq\|u_0\|_{L^\infty}$.\medskip

Observe that $\widehat{J}_k=(\widehat{J})^k\in L^\I(\R^n)$ with $|\widehat{J}_k(\xi)|\leq\|J\|_{L^1}^k=1$. By (\ref{Fourier:Jsigma}) and the Riemann-Lebesgue lemma, there are positive constants $\r_0,\D,D$ (depending only on $J$) such that
\begin{align*}
\begin{cases}
|\widehat{J}(\xi)|\leq1-D|\xi|^\sigma&\mbox{for $|\xi|\leq \r_0$},\\
\vspace{-10pt}\\
|\widehat{J}(\xi)|\leq1-\D&\mbox{for $|\xi|>\r_0$}.
\end{cases}
\end{align*}
Observe that $\widehat{\CJ_ku_0}=\widehat{J}_k\widehat{u}_0\in L^1(\R^n)$, hence $\CJ_ku_0\in L^\infty(\R^n)$.\medskip

For $|\xi|\leq \r_0\theta_k$, where $\theta_k:=k^{1/\sigma}$, we have $|\xi|/\theta_k\leq\r_0$ so
\begin{align*}
\left|\widehat{J}_k\left(\frac{\xi}{\theta_k}\right)\right|&=\left|\widehat{J}\left(\frac{\xi}{\theta_k}\right)\right|^{k}\\
&\leq\left(1-D\frac{|\xi|^\sigma}{k}\right)^{k}\\
&\leq e^{-D|\xi|^{\sigma}},
\end{align*}
where we have used the elementary inequality 
\begin{align}\label{Tool:linexp}
1-Dx/k\leq e^{-Dx/k}\quad\mbox{for all $x\geq0$ and $D,k>0$}.
\end{align}

On the other hand, if $|\xi|>\r _0\theta_k$ then
\[
\left|\widehat{J}_k\left(\frac{\xi}{\theta_k}\right)\right|\leq\left(1-\D\right)^{k}.
\]
Since $(k^{-n/\sigma})^{1/k}\to1$ as $k\to\I$, $\exists$ $N>0$ (depending only upon $n,J$) such that 
\[
(1-\D)^{k}\leq k^{-n/\sigma}=\theta_k^{-n}\quad\mbox{for all $k\geq N$}.
\]

Now fix $k\geq N$. We estimate the integral
\begin{align*}
\int_{\R^n}\left|\widehat{\CJ_ku_0}\left(\frac{\xi}{\theta_k}\right)\right|d\xi&=\int_{|\xi|\leq \r_0\theta_k}\left|\widehat{\CJ_ku_0}\left(\frac{\xi}{\theta_k}\right)\right|d\xi+\int_{|\xi|>\r _0\theta_k}\left|\widehat{\CJ_ku_0}\left(\frac{\xi}{\theta_k}\right)\right|d\xi=:I_1+I_2.
\end{align*}
Then we have
\begin{align*}
&I_1=\int_{|\xi|\leq \r_0\theta_k}\left|\widehat{J}_k\left(\frac{\xi}{\theta_k}\right)\right|\left|\widehat{u}_0\left(\frac{\xi}{\theta_k}\right)\right|d\xi\\
&\hphantom{I_1}\leq\int_{|\xi|\leq \r_0\theta_k}e^{-D|\xi|^\sigma}\|\widehat{u}_0\|_{L^\I}d\xi\\
&\hphantom{I_1}\leq\|u_0\|_{L^1}\int_{\R^n}e^{-D|\xi|^\sigma}d\xi=:C_1\|u_0\|_{L^1}<\I,
\end{align*}
and
\begin{align*}
&I_2\leq\int_{|\xi|>\r_0\theta_k}(1-\D)^{k}\left|\widehat{u}_0\left(\frac{\xi}{\theta_k}\right)\right|d\xi\\
&\hphantom{I_2}\leq\int_{\R^n}\theta_k^{-n}\left|\widehat{u}_0\left(\frac{\xi}{\theta_k}\right)\right|d\xi\\
&\hphantom{I_2}\leq\|\widehat{u}_0\|_{L^1}=:C_2\|\widehat{u}_0\|_{L^1}<\I.
\end{align*}
Combining the estimates for $I_1,I_2$ we get that
\[
\int_{\R^n}\left|\widehat{\CJ_ku_0}\left(\frac{\xi}{\theta_k}\right)\right|d\xi\leq C(\|u_0\|_{L^1}+\|\widehat{u}_0\|_{L^1}),
\]
for some constant $C>0$ depends only on $n,J$.\medskip

By the Fourier inversion formula we have
\begin{align*}
k^{n/\sigma}\|\CJ_ku_0\|_{L^\infty}&\leq C_{n}\theta_k^n\|\widehat{\CJ_ku_0}\|_{L^1}\\
%&=\left\{k^{n/\sigma}\int_{\R^n}\left|\widehat{\CJ_ku_0}(\xi)\right|^rd\xi\right\}^{1/r}\\
&=C_{n}\int_{\R^n}\left|\widehat{\CJ_ku_0}\left(\frac{\xi}{\theta_k}\right)\right|d\xi\\
&\leq C_{n,J}(\|u_0\|_{L^1}+\|\widehat{u}_0\|_{L^1}),
\end{align*}
therefore we obtain that
\[
\|\CJ_ku_0\|_{L^\infty}\leq C_{n,J}(\|u_0\|_{L^1}+\|\widehat{u}_0\|_{L^1})k^{-n/\sigma}\quad\forall\,k\geq N.
\]
This completes the proof of the theorem.\QED

\end{proof}

\begin{remark}

In the case that $\sigma=2$, the estimate (\ref{Est1:ThmStableJk}) can be sharpen with the dependence on the initial condition on the right hand side removed. As was observed in \citep{ChasseigneEtal06}, if the Fourier transform of $J$ satisfies the second part of (\ref{Fourier:Jsigma}) with $\sigma=2$, then $J$ has the second moment, so the Local Limit Theorem can be implied to get the sharper estimate.

\end{remark}

The following result was proved in \citep{ChasseigneEtal06}.

\begin{corollary}\label{Cor:DecaySigma}

Assume $J$ and $u_0$ satisfy (\ref{Fourier:Jsigma}). Then the solution $u(t)$ of (\ref{Eqn:main}) satisfies $u(t)\in L^p(\R^n)$ for all $t\geq0$, for any $1\leq p\leq\infty$. Furthermore,
\[
\|u(t)\|_{L^p}\leq C\left(\|u_0\|_{L^1}+\|\widehat{u}_0\|_{L^1}\right)t^{-\frac{n}{\sigma}(1-\frac{1}{p})}\quad\mbox{as $t\to\I$},
\]
where $C>0$ is a constant depending on $n,J$.

\end{corollary}

\begin{proof}

By Theorem \ref{Thm:StableJk}, we have
\[
\CJ_ku_0\in L^p(\R^n),\quad k=1,\ldots,N-1
\]
and
\[
\|\CJ_ku_0\|_{L^p}\leq R_k:=C(\|u_0\|_{L^1}+\|\widehat{u}_0\|_{L^1})k^{-\frac{n}{\sigma}(1-\frac{1}{p})},\quad k=N,N+1,\ldots,
\]
i.e.\ (\ref{Hyp:H3}) is true. According to part (ii) of Theorem \ref{Thm:GenDecay}, we then have
\[
\|u(t)\|_{L^p}\leq C(\|u_0\|_{L^1}+\|\widehat{u}_0\|_{L^1})t^{-\frac{n}{\sigma}(1-\frac{1}{p})}\quad\forall\,t\geq t_0>0,
\]
which proves the theorem.\QED

\end{proof}

For the borderline case that $\widehat{J}$ has the asymptotic expansion with logarithmic perturbation:
\[
\widehat{J}(\xi)\sim1-A|\xi|^2(\ln1/|\xi|)\quad\mbox{as $\xi\to0$},
\]
was considered in the last section of \cite{ChasseigneEtal06}. (If $n=1$, this case corresponds to the Fourier transform of $J(x)\sim1/|x|^3$ as $|x|\to\infty$.) In this case, it was shown that the solution $u(t)$ to (\ref{Eqn:main}) has the asymptotic decay
\[
\|u(t)\|_{L^\infty}\lesssim(t\ln t)^{-n/2}\quad\mbox{as $t\to\infty$},
\]
for all $u_0\in L^1(\R^n)$, $\widehat{u}_0\in L^1(\R^n)$. We present the following generalization.

\begin{theorem}\label{Thm:Jksigmaln}

Assume
\begin{align}\label{Fourier:Jkln}\tag{H5}
\begin{cases}
\displaystyle
J=J(|x|)\geq0,\quad \chi_0=\|J\|_{L^1}=1,\\
\vspace{-10pt}\\
\displaystyle
\widehat{J}(\xi)=1-A|\xi|^\sigma(\ln1/|\xi|)^\mu+o\left(|\xi|^\sigma(\ln1/|\xi|)^\mu\right)\quad\mbox{as $\xi\to0$,}\\
\vspace{-10pt}\\
\displaystyle
\hspace{4cm}\mbox{where $\sigma\in(0,2],\mu\in\R$, $A>0$,}\\
\vspace{-10pt}\\
\displaystyle
u_0\in L^1(\R^n),\,\,\widehat{u}_0\in L^1(\R^n).
\end{cases}
\end{align}
Then,
\[
\CJ_ku_0\in L^p(\R^n)\quad\mbox{for all $k$, for any $1\leq p\leq\infty$},
\]
and, furthermore, there is a positive integer $N=N(n,J)$ such that
\begin{align}
\|\CJ_ku_0\|_{L^p}\leq C\left(\|u_0\|_{L^1}+\|\widehat{u}_0\|_{L^1}\right)(k(\ln k)^\mu)^{-\frac{n}{\sigma}(1-\frac{1}{p})}\quad\forall\,k\geq N,
\end{align}
where $C>0$ is a constant depending only on $n,J$.
\end{theorem}

\begin{proof}

Again it suffices to prove the results for $p=\infty$, the rest will follow from interpolation. Let $u_0\in L^1(\R^n)$ be such that $\widehat{u}_0\in L^1(\R^n)$. By the asymptotic behavior of $\widehat{J}$ and the Riemann-Lebesgue lemma, there are $\D>0,0<\r_0<1$ such that
\[
|\widehat{J}(\xi)|\leq
\begin{cases}
\displaystyle
1-D|\xi|^\sigma\left|\ln\frac{1}{|\xi|}\right|^\mu&|\xi|\leq \r_0,\\
\vspace{-10pt}\\
\displaystyle
1-\D&|\xi|>\r_0.
\end{cases}
\]
The case $\mu=0$ was considered before. So assume $\mu\neq0$.\medskip

\textbf{Case I: $\mu>0$.}  Let 
\[
\theta_k=(k(\ln k)^\mu)^{1/\sigma}\quad\mbox{and}\quad\r_k=\r_0k^{-\varepsilon},\quad0<\varepsilon<\min\{\sigma,1/\sigma\}.
\]
Note that $\r_k\theta_k\to\infty$ as $k\to\infty$ by the regular variation of $\r_k\theta_k$.\medskip

If $|\xi|\leq \r_k\theta_k$ then $\ln(\theta_k/|\xi|)\geq\ln(1/\r_k)=\varepsilon\ln k+\ln(1/\r_0)\geq\varepsilon\ln k$, and hence 
\begin{align*}
\left|\widehat{J}_k\left(\frac{\xi}{\theta_k}\right)\right|&\leq\left(1-D\frac{|\xi|^\sigma}{k(\ln k)^\mu}\left(\ln\frac{\theta_k}{|\xi|}\right)^\mu\right)^k\\
&\leq\left(1-D_1\frac{|\xi|^\sigma}{k}\right)^k,\\
&\leq e^{-D_1|\xi|^\sigma}.
\end{align*}
If $\r_k\theta_k\leq|\xi|\leq \r_0\theta_k$ then $\ln(\theta_k/|\xi|)\geq\ln(1/\r_0)>0$ and 
\[
\frac{|\xi|^\varepsilon}{(\ln k)^\mu}\geq\frac{(\r_k\theta_k)^\varepsilon}{(\ln k)^\mu}=\r_0^\varepsilon k^{((1/\sigma)-\varepsilon)\varepsilon}(\ln k)^{(\mu\varepsilon/\sigma)-\mu}\geq c>0,
\]
for all $k\geq 2$. Hence
\begin{align*}
\left|\widehat{J}_k\left(\frac{\xi}{\theta_k}\right)\right|&\leq\left(1-D\frac{|\xi|^{\sigma-\varepsilon}|\xi|^\varepsilon}{k(\ln k)^\mu}\left(\ln\frac{1}{\r_0}\right)^\mu\right)^k\\
&\leq\left(1-D_2\frac{|\xi|^{\sigma-\varepsilon}}{k}\right)^k\\
&\leq e^{-D_2|\xi|^{\sigma-\varepsilon}},
\end{align*}
by (\ref{Tool:linexp}).\medskip

Finally since $(k(\ln k)^\mu)^{1/k}\to1$ as $k\to\I$, there is $N\geq2$ such that
\[
(1-\D)^k\leq(k(\ln k)^\mu)^{-n/\sigma}=\theta_k^{-n}\quad\mbox{for all $k\geq N$},
\]
which give
\begin{align}\label{Tmp:powerln}
|\xi|>\r_0\theta_k,\,\,k\geq N\quad\Rightarrow\quad\left|\widehat{J}_k\left(\frac{\xi}{\theta_k}\right)\right|&\leq\theta_k^{-n}.
\end{align}
Now, for all $k\geq N$, we have by the preceding calculations that
\begin{align*}
\int_{\R^n}\left|\widehat{\CJ_ku_0}\left(\frac{\xi}{\theta_k}\right)\right|d\xi&=\int_{\R^n}\left|\widehat{J}_k\left(\frac{\xi}{\theta_k}\right)\cdot\widehat{u}_0\left(\frac{\xi}{\theta_k}\right)\right|d\xi\\
&=\int_{|\xi|\leq \r_k\theta_k}\left|\cdots\right|d\xi+\int_{\r_k\theta_k<|\xi|\leq \r_0\theta_k}|\cdots|d\xi
+\int_{|\xi|>\r_0\theta_k}|\cdots|d\xi\\
&\leq\|\widehat{u}_0\|_{L^\I}\int_{\R^n}e^{-D_1|\xi|^\sigma}+e^{-D_2|\xi|^{\sigma-\varepsilon}}d\xi+\int_{\R^n}\theta_k^{-n}\left|\widehat{u}_0\left(\frac{\xi}{\theta_k}\right)\right|d\xi\\
&\leq C(\|u_0\|_{L^1}+\|\widehat{u}_0\|_{L^1})<\I.
\end{align*}
By Hausdorff-Young inequality, we obtain that
\begin{align*}
(k(\ln k)^\mu)^{n/{\sigma}}\|\CJ_ku_0\|_{L^\infty}&\lesssim\theta_k^{n}\int_{\R^n}\left|\widehat{\CJ_ku_0}(\xi)\right|d\xi\\
&=\int_{\R^n}\left|\widehat{\CJ_ku_0}\left(\frac{\xi}{\theta_k}\right)\right|d\xi\leq C(\|u_0\|_{L^1}+\|\widehat{u}_0\|_{L^1}).
\end{align*}
Hence we obtain
\[
\|\CJ_ku_0\|_{L^\infty}\leq C(\|u_0\|_{L^1}+\|\widehat{u}_0\|_{L^1})(k(\ln k)^\mu)^{-n/\sigma},
\]
which is the desired estimate when $\mu>0$.\medskip

\textbf{Case II: $\mu<0$.} If $|\xi|\leq1$ then we use the estimate $|\widehat{J}_k(\xi/\theta_k)|\leq1$. If $1<|\xi|\leq\r_0\theta_k$ then $\theta_k/|\xi|\leq\theta_k$ and $\ln\theta_k=(\ln k)/\sigma+\mu/\sigma\ln\ln k\leq(\ln k)/\sigma$ for all $k\geq3$, hence
\begin{align*}
\left|\widehat{J}_k\left(\frac{\xi}{\theta_k}\right)\right|&\leq\left(1-D\frac{|\xi|^\sigma}{k(\ln k)^\mu}\left(\ln\frac{\theta_k}{|\xi|}\right)^\mu\right)^k\\
&\leq\left(1-D\frac{|\xi|^\sigma}{k(\ln k)^\mu}(\ln\theta_k)^\mu\right)^k\\
&\leq\left(1-D_1\frac{|\xi|^\sigma}{k}\right)^k\leq e^{-D_1|\xi|^\sigma},
\end{align*}
by (\ref{Tool:linexp}). We apply the estimate (\ref{Tmp:powerln}) above for $|\xi|>\r_0\theta_k$. Then we obtain
\begin{align*}
\int_{\R^n}\left|\widehat{\CJ_ku_0}\left(\frac{\xi}{\theta_k}\right)\right|d\xi&=\int_{|\xi|\leq 1}\left|\cdots\right|d\xi+\int_{1<|\xi|\leq \r_0\theta_k}|\cdots|d\xi
+\int_{|\xi|>\r_0\theta_k}|\cdots|d\xi\\
&\leq\|\widehat{u}_0\|_{L^\I}\left\{\omg_n+\int_{\R^n}e^{-D_1|\xi|^\sigma}d\xi\right\}+\int_{\R^n}\theta_k^{-n}\left|\widehat{u}_0\left(\frac{\xi}{\theta_k}\right)\right|d\xi\\
&\leq C(\|u_0\|_{L^1}+\|\widehat{u}_0\|_{L^1})<\I.
\end{align*}
The remaining now follows by the same argument as the case $\mu>0$. \QED

\end{proof}

\begin{theorem}\label{Thm:DecayLn}

Assume (\ref{Fourier:Jkln}). Then the solution $u(t)$ of (\ref{Eqn:main}) satisfies $u(t)\in L^p(\R^n)$ for any $1\leq p\leq\I$. Furthermore,
\[
\|u(t)\|_{L^p}\leq C(\|u_0\|_{L^1}+\|\widehat{u}_0\|_{L^1})(t(\ln t)^\mu)^{-\frac{n}{\sigma}(1-\frac{1}{p})}\quad\mbox{as $t\to\I$},
\]
where $C>0$ is a constant depending only on $n,J$.

\end{theorem}

\begin{proof}
Simply apply Theorem \ref{Thm:Jksigmaln} and part (ii) of Theorem \ref{Thm:GenDecay}.\QED

\end{proof}

\begin{remark}

It can be seen easily that the argument used in the proof of the preceding theorem can be applied to $J\in L^1(\R^n)$ having the asymptotic expansion
\[
\widehat{J}(\xi)=1-A|\xi|^\sigma(\ln1/|\xi|)^{\mu_1}(\ln_21/|\xi|)^{\mu_2}\cdots(\ln_m1/|\xi|)^{\mu_m}+l.o.t\quad\mbox{as $\xi\to0$},
\]
where $\ln_k=\ln\circ\cdots\circ\ln$ ($k$ terms), and we get the asymptotic behavior
\[
\|u(t)\|_{L^p}\lesssim(t(\ln t)^{\mu_1}\cdots(\ln_mt)^{\mu_m})^{-\frac{n}{\sigma}(1-\frac{1}{p})}\quad\mbox{as $t\to\I$},
\]
provided $u_0\in L^1(\R^n),\widehat{u}_0\in L^1(\R^n)$, $1\leq p\leq\infty$, $0<\sigma\leq2,\mu_1,\ldots,\mu_m\in\R$.

\end{remark}

\section{Nonlocal equations with prescribed decay}

In this section we present condition on $J$ which guarantees that the solution to (\ref{Eqn:main}) has the decay rate given by a regular varying function with negative index:
\[
\|u(t)\|_{L^p}\lesssim (tL(t))^{-\beta}\quad\mbox{as $t\to\infty$},
\]
where $\beta>0$ and $L:(0,\infty)\to(0,\infty)$ is slowly varying. By the smooth variation theorem for regular varying functions, we can assume without loss of generality that $L$ is smooth; in particular, it is continuous. We need to impose an important hypothesis:
\begin{align}\label{Hyp:Mono}
\mbox{$L$ is eventually monotone, i.e.\ $\exists\,N_0>0$ such that $L$ is monotone on $[N_0,\infty)$.}
\end{align}

The main hypothesis is
\begin{align}\label{Hyp:H6}\tag{H6}
\begin{cases}
\displaystyle
J=J(|x|)\geq0,\quad\chi_0=\|J\|_{L^1}=1,\\
\vspace{-10pt}\\
\displaystyle
\widehat{J}(\xi)=1-A|\xi|^{\sigma}L\left(|\xi|^{-\gamma}\right)+o\left(|\xi|^{\sigma}L\left(|\xi|^{-\gamma}\right)\right)\quad\mbox{as $|\xi|\to0$},\\
\vspace{-10pt}\\
\displaystyle
\hspace{3cm}\mbox{where $\sigma=\frac{n}{\beta}(1-\frac{1}{p}),\gamma>0$, $1<p\leq\frac{n}{(n-2\beta)_+}$},\\
\vspace{-10pt}\\
\displaystyle
u_0\in L^1(\R^n),\,\,\widehat{u}_0\in L^1(\R^n).
\end{cases}
\end{align}
Note that $0<\sigma\leq2$.

\begin{theorem}

Let $L:(0,\infty)\to(0,\infty)$ be a slowly varying function satisfying (\ref{Hyp:Mono}), $\beta>0$. Assume (\ref{Hyp:H6}) with
\[
\gamma>\sigma\quad\mbox{if $L$ is eventually increasing},\quad\gamma=\sigma\quad\mbox{if $L$ is eventually decreasing}.
\]
Then there is a positive integer $N=N(n,J)$ such that $\CJ_ku_0\in L^p(\R^n)$ for all $k$ and
\[
\|\CJ_ku_0\|_{L^p}\leq C\left(\|u_0\|_{L^1}+\|\widehat{u}_0\|_{L^1}\right)(kL(k))^{-\beta}\quad\forall\,k\geq N,
\]
where $C>0$ is a constant. Moreover, in this case, the solution $u(t)$ of (\ref{Eqn:main}) satisfies 
\[
\|u(t)\|_{L^p}\lesssim(tL(t))^{-\beta}\quad\mbox{as $t\to\infty$}.
\]

\end{theorem}

\begin{proof}
It is obvious that $\CJ_ku_0\in L^q(\R^n)$ for all $k$ and $1\leq q\leq\infty$. 
Since we are interested in the behavior of solution of (\ref{Eqn:main}) as $t\to\infty$ and $N$ can be chosen arbitrarily (independent of $t$), the values of $L$ on $(0,N_0)$ is irrelevant. By redefining the function, we can assume that $L$ is monotone on $(0,\infty)$.
If $\lim_{s\to\infty}L(s)$ is a finite positive number, then we have nothing to prove. So we will assume 
\begin{align}
\lim_{s\to\infty}L(s)=\begin{cases}
\infty&\mbox{if $L$ is increasing},\\
0&\mbox{if $L$ is decreasing}.
\end{cases}
\end{align}
By the hypothesis (\ref{Hyp:H6}) of $J$ and the Riemann-Lebesgue lemma, there are $\delta,D,\r_0>0$ such that
\begin{align*}
|\widehat{J}(\xi)|\leq\begin{cases}
\displaystyle
1-D|\xi|^\sigma L(|\xi|^{-\gamma})&|\xi|\leq \r_0,\\
\vspace{-10pt}\\
\displaystyle
1-\delta&|\xi|>\r_0,
\end{cases}
\end{align*}

\textbf{Case I: $L(s)\to\infty$.} For this case, $\gamma>\sigma$. We define
\[
\theta_k=(kL(k))^{1/\sigma}\quad\mbox{and}\quad\r_k=\r_0k^{-1/\gamma},\quad k=1,2,\ldots
\]

If $|\xi|\leq\r_k\theta_k$ then $(|\xi|/\theta_k)^{-\gamma}\geq\r_k^{-\gamma}=\r_0^{-\gamma}k$. Since $L$ is increasing and is slowly varying, we have
\[
L\left(\left(|\xi|/\theta_k\right)^{-\gamma}\right)\geq L(\r_0^{-\gamma}k)\sim L(k),\quad\mbox{as $k\to\infty$}.
\]
Thus for all $k$ sufficiently large, we have
\begin{align*}
\left|\widehat{J}_k\left(\frac{\xi}{\theta_k}\right)\right|&\leq\left(1-D\frac{|\xi|^\sigma}{kL(k)}L\left(\left(|\xi|/\theta_k\right)^{-\gamma}\right)\right)^k\\
&\leq\left(1-D_1\frac{|\xi|^\sigma}{k}\right)^k\leq e^{-D_1|\xi|^\sigma},
\end{align*}
by (\ref{Tool:linexp}).\medskip

Let $0<\varepsilon<\sigma$. 
If $\r_k\theta_k<|\xi|\leq\r_0\theta_k$ then $(|\xi|/\theta_k)^{-\gamma}\geq\r_0^{-\gamma}$ and 
\[
\frac{|\xi|^\varepsilon}{L(k)}\geq\frac{(\r_k\theta_k)^\varepsilon}{L(k)}=\r_0^{\varepsilon}k^{(1/\sigma-1/\gamma)\varepsilon}L(k)^{\varepsilon/\sigma-1}\geq c>0,
\]
since $k^{(1/\sigma-1/\gamma)\varepsilon}L(k)^{\varepsilon/\sigma-1}$ is regular varying with positive index. Hence
\begin{align*}
\left|\widehat{J}_k\left(\frac{\xi}{\theta_k}\right)\right|&\leq\left(1-D\frac{|\xi|^{\sigma-\varepsilon}|\xi|^\varepsilon}{kL(k)}L(\r_0^{-\gamma})\right)^k\\
&\leq\left(1-D_2\frac{|\xi|^{\sigma-\varepsilon}}{k}\right)^k\\
&\leq e^{-D_2|\xi|^{\sigma-\varepsilon}}.
\end{align*}

Finally, if $|\xi|>\r_0\theta_k$ then we have
\begin{align*}
\left|\widehat{J}_k\left(\frac{\xi}{\theta_k}\right)\right|&\leq(1-\delta)^k\leq(kL(k))^{-n/\sigma}=\theta_k^{-n},
\end{align*}
for all $k$ sufficiently large. Here we have used that $L$ is slowly varying, so $(\alpha_1 \sqrt{k})^{1/k}\leq(kL(k))^{1/k}\leq (\alpha_2k^2)^{1/k}$ for some constants $\alpha_1,\alpha_2>0$, and
\[
\lim_{k\to\infty}(\alpha_1\sqrt{k})^{1/k}=\lim_{k\to\infty}(\alpha_2k^2)^{1/k}=1\quad\therefore\,(kL(k))^{1/k}\to1,
\]
as $k\to\infty$.\medskip

Combining the above calculations we now get, for $k$ sufficiently large, that
\begin{align*}
\int_{\R^n}\left|\widehat{\CJ_ku_0}\left(\frac{\xi}{\theta_k}\right)\right|d\xi&=\int_{\R^n}\left|\widehat{J}_k\left(\frac{\xi}{\theta_k}\right)\cdot\widehat{u}_0\left(\frac{\xi}{\theta_k}\right)\right|d\xi\\
&=\int_{|\xi|\leq \r_k\theta_k}\left|\cdots\right|d\xi+\int_{\r_k\theta_k<|\xi|\leq \r_0\theta_k}|\cdots|d\xi
+\int_{|\xi|>\r_0\theta_k}|\cdots|d\xi\\
&\leq\|\widehat{u}_0\|_{L^\infty}\int_{\R^n}e^{-D_1|\xi|^\sigma}+e^{-D_2|\xi|^{\sigma-\varepsilon}}d\xi+\theta_k^{-n}\int_{\R^n}|\widehat{u}_0(\xi/\theta_k)|d\xi\\
&=C(\|u_0\|_{L^1}+\|\widehat{u}_0\|_{L^1})<\infty.
\end{align*}
Applying Hausdorff-Young inequality then we get
\begin{align*}
(kL(k))^{n/\sigma}\|\CJ_ku_0\|_{L^\infty}&\lesssim\theta_k^n\int_{\R^n}\left|\widehat{\CJ_ku_0}(\xi)\right|d\xi\\
&=\int_{\R^n}\left|\widehat{\CJ_ku_0}\left(\frac{\xi}{\theta_k}\right)\right|d\xi\leq C(\|u_0\|_{L^1}+\|\widehat{u}_0\|_{L^1}).
\end{align*}
Therefore
\begin{align*}
\|\CJ_ku_0\|_{L^\infty}\leq C(\|u_0\|_{L^1}+\|\widehat{u}_0\|_{L^1})(kL(k))^{-n/\sigma}
\end{align*}
By interpolation we also get
\begin{align*}
\|\CJ_ku_0\|_{L^p}\leq C(kL(k))^{-\frac{n}{\sigma}(1-\frac{1}{p})}=C(kL(k))^{-\beta}
\end{align*}
Using Theorem \ref{Thm:GenDecay} (ii), it follows that
\[
\|u(t)\|_{L^p}\lesssim(tL(t))^{-\beta}\quad\mbox{as $t\to\infty$}.
\]

\textbf{Case II: $L(s)\to0$.} For this case $\gamma=\sigma$. We employ a similar argument as in the proof of Theorem \ref{Thm:Jksigmaln}. Use $\theta_k=(kL(k))^{1/\sigma}$ as in the previous case. Let us split $\R^n$ into
\begin{align*}
\{|\xi|\leq1\},\quad\{1<|\xi|\leq\r_0\theta_k\},\quad\{|\xi|>\r_0\theta_k\}.
\end{align*}
If $|\xi|\leq1$, we employ the estimate $|\widehat{J}_k(\xi/\theta_k)|\leq1$. Assume $1<|\xi|\leq\r_0\theta_k$. Then
\[
L((\theta_k/|\xi|)^\gamma)\geq L(\theta_k^\gamma)=L(kL(k)),
\]
where we have used that $L$ is now decreasing.
Since $L(k)\to0$ as $k\to\infty$ and $L$ is decreasing, it follows that
\[
L((\theta_k/|\xi|)^\gamma)\geq L(k)\quad\mbox{for all $k$ sufficiently large}.
\]
So, in this case,
\begin{align*}
\left|\widehat{J}_k\left(\frac{\xi}{\theta_k}\right)\right|&\leq\left(1-D\frac{|\xi|^\sigma}{kL(k)}L\left((\theta_k/|\xi|)^\gamma\right)\right)^k\\
&\leq\left(1-D\frac{|\xi|^\sigma}{k}\right)^k\leq e^{-D|\xi|^\sigma}.
\end{align*}
Finally, if $|\xi|>\r_0\theta_k$, then we have, as in the previous case, that
\begin{align*}
\left|\widehat{J}_k\left(\frac{\xi}{\theta_k}\right)\right|&\leq\theta_k^{-n},
\end{align*}
for all $k$ sufficiently large.\medskip

Combining the preceding calculations, we get
\begin{align*}
\int_{\R^n}\left|\widehat{\CJ_ku_0}\left(\frac{\xi}{\theta_k}\right)\right|d\xi&=\int_{|\xi|\leq1}|\cdots|d\xi+\int_{1<|\xi|\leq\r_0\theta_k}|\cdots|d\xi+\int_{|\xi|>\r_0\theta_k}|\cdots|d\xi\\
&\leq C\|\widehat{u}_0\|_{L^\infty}\left\{\omg_n+\int_{\R^n}e^{-D|\xi|^\sigma}d\xi\right\}+\int_{\R^n}\theta_k^{-n}|\widehat{u}_0(\xi/\theta_k)|d\xi\\
&\leq C(\|u_0\|_{L^1}+\|\widehat{u}_0\|_{L^1})<\infty
\end{align*}
hence
\[
\|\CJ_ku_0\|_{L^\infty}\lesssim\|\CJ_ku_0\|_{L^1}\lesssim(kL(k))^{1/\sigma},
\]
for all $k$ sufficiently large. Invoking Theorem \ref{Thm:GenDecay} (ii), we obtain the desired asymptotic behavior of $u(t)$.\QED

\end{proof}


\begin{thebibliography}{00}

\bibitem{Alfaro17}
M.\ Alfaro, Fujita blow up phenomena and hair trigger effect: the role of dispersal tails, Ann.\ Inst.\ H.\ Poincar
'e Anal.\ Non Lin\'eaire, 34(5), 2017, 1309--1327.


\bibitem{AndreuEtal10}
F.\ Andreu, J.M.\ Maz\'on, J.D.\ Rossi, J.\ Toledo, Nonlocal Diffusion Problems, Mathematical Surveys and Monographs, 165, American Mathematical Society, Providence; Real Sociedad Matem\'atica Espa\~{n}ola, Madrid, 2010.

\bibitem{BatesEtal99}
P.W.\ Bates, A.\ Chmaj, A discrete convolution model for phase transitions, Arch.\
Rational Mech.\ Anal.\ 150(4), 281–-305 (1999).

\bibitem{BealsWong10}
R.\ Beals, R.\ Wong, Special functions, A graduate text, Cambridge Studies in Advanced Mathematics 126, CUP, Cambridge, 2010.

\bibitem{BinghamEtal83}
N.H.\ Bingham, J.\ Hawkes, Some 
limit theorems for occupation times, In: Probability, Statistics \& Analysis (eds.\ 
J.F.C.\ Kingman \& G.E.H.\ Reuter), pp.\ 46--62.\ London Math.\ Soc.\ Lecture 
Notes, vol.\ 79, CUP.\ 1983. 


\bibitem{Bingham89}
N.\ Bingham, C.\ Goldie, J.\ Teugels, Regular Variation, Cambridge University Press (1989).

\bibitem{BrandleEtal11}
C.\ Br\"andle, E.\ Chasseigne, R.\ Ferreira, Unbounded solutions of the nonlocal heat equation, Commun.\ Pure Appl.\ Anal.\ 10 (2011), no. 6, 1663–-1686

\bibitem{BrascampLieb76}
H.J.\ Brascamp, E.H.\ Lieb, Best constants in Young’s inequality, its converse
and its generalization to more than three functions, Adv.\ Math.\ 20 (1976) 151--173.

\bibitem{Caravenna12}
F.\ Caravenna, A note on directly Riemann integrable functions, arXiv: 12102361v1 (2012).

\bibitem{ChasseigneEtal06}
E.\ Chasseigne, M.\ Chaves, J.D.\ Rossi,  Asymptotic behavior for nonlocal diffusion equations, J.\ Math.\ Pures Appl.\ 86 (2006) 271–-291.


\bibitem{ChasseigneEtal14}
E.\ Chasseigne, P.\ Felmer, J.D.\ Rossi, E.\ Topp, Fractional decay bounds for nonlocal zero order
heat equations, Bull London Math.\ Soc., 46 (2014), 943--952.


%\bibitem{CortazarEtal07}
%C.\ Cort\'azar, J.\ Coville, M.\ Elgueta, S.\ Mart\'inez,  
%A nonlocal inhomogeneous dispersal process, J. Differential Equations, 241 (2007) 332--358.

%\bibitem{CortazarEtal15}
%C.\ Cort\'azar, M.\ Elgueta, J.\ Garc\'ia-Meli\'an, S.\ Mart\'inez, Finite mass solutions for a
%nonlocal inhomogeneous dispersal equation, Discrete Cont.\ Dyn.\ Syst.\ 35 (2015), 1409-–1419.

%\bibitem{CortazarEtal16}
%C.\ Cort\'azar, M.\ Elgueta, J.\ Garc\'ia-Meli\'an, S.\ Mart\'inez, An inhomogeneous nonlocal diffusion problem with unbounded
%steps, J.\ Evol.\ Equ.\ 16 (2016), 209–-232.


\bibitem{Feller71}
W.\ Feller, An Introduction to Probability Theory and Its Applications, vol.\ 2, 2nd ed. (1971), John Wiley and Sons, New York.

\bibitem{Fife03}
P.\ Fife, Some nonclassical trends in parabolic and parabolic-like evolutions, Trends in Nonlinear Analysis, 153-–191. Springer, Berlin, 2003.

\bibitem{GilboaEtal08}
G.\ Gilboa, S.\ Osher, Nonlocal operators with application to image processing, Multiscale Model.\ Simul.\ 7(3), 1005-–1028 (2008).


\bibitem{IgnatEtal08}
L.I.\ Ignat, J.D.\ Rossi, Refined asymptotic expansions for nonlocal diffusion equations, J.\ Evol.\ Equ.\ 8(4) (2008), 617-–629.


\bibitem{IgnatRossi09}
L.\ Ignat and J.D.\ Rossi, Decay estimates for nonlocal problems via energy methods, J.\ Math.\ Pures Appl.\ (9) 92 (2009) 163-–187.


\bibitem{Karamata30}
J.\ Karamata, Sur un mode de croissance r\'eguli\`ere des fonctions, Mathematica (Cluj) 4 (1930) 38–-53.

%\bibitem{KhomrutaiNA15}
%S.\ Khomrutai, Global and blow-up solutions of superlinear pseudoparabolic equations with unbounded coefficient, 122 July 2015, 192--214, Nonlinear Analysis, Theory, Methods, \& Applications.

\bibitem{KonMolVain17}
Y.\ Kondratiev, S.\ Molchanov, B.\ Vainberg, Spectral analysis of nonlocal Schrödinger operators, J.\ Funct.\ Anal.\, 273 (2017), 1020--1048.


\bibitem{LiebLoss01}
E.H.\ Lieb, M.\ Loss, Analysis 
(second ed.), Grad.\ Stud.\ Math., vol.\  14, Amer.\ Math.\ Soc., Providence, RI (2001)


\bibitem{TricErde51}
F.G.\ Tricomi, A.\ Erd\'elyi, 
The asymptotic expansion of a ratio of gamma functions, Pacific J.\ Math., 1 (1) (1951), 133--142.


\end{thebibliography}
\end{document}